\documentclass[12pt]{article}

\usepackage[margin=1in]{geometry}
\usepackage{diagbox}
\usepackage{mathtools}
\usepackage{bbm}
\usepackage{latexsym}
\usepackage{epsfig}
\usepackage{amsmath,amsthm,amssymb,enumerate}
\usepackage{hyperref}
\usepackage[usenames,dvipsnames]{xcolor}

\usepackage{color}

\def\cH{{\mathcal H}}
\def\nn{\nonumber}
\def\a{\alpha}  \def\d{\delta} 
\def\e{\varepsilon} \def\f{\phi}   
  \def\k{\kappa}

\newtheorem{theorem}{Theorem}
\newtheorem{lemma}[theorem]{Lemma}

\newtheorem{claim}{Claim}
\newtheorem{definition}{Definition}

\newcommand{\bfrac}[2]{\left(\frac{#1}{#2}\right)}

\newcommand{\set}[1]{\left\{#1\right\}}
\def\sm{\setminus}

\def\E{\mathbb{E}}

\def\Pr{\mathbb{P}}

\newcommand{\ignore}[1]{}

\def\cH{{\mathcal H}}

\newcommand{\beq}[2]{\begin{equation}\label{#1}#2\end{equation}}
\newcommand{\mults}[1]{\begin{multline*}#1\end{multline*}}

\def\nn{\nonumber}

\newcommand{\brac}[1]{\left(#1\right)}

\newcommand{\bin}{\textrm{Bin}}
\renewcommand{\P}{\mathbb{P}}
\def\E{\mathbb{E}}
\newcommand{\rbrac}[1]{\left(#1\right)}
\newcommand{\sbrac}[1]{\left[ #1\right]}
\newcommand{\cbrac}[1]{\left\{ #1\right\}}
\newcommand{\abrac}[1]{\left| #1\right|}
\newcommand{\mc}[1]{\mathcal{#1}}
\newcommand{\tbf}[1]{\textbf{#1}}

\newcounter{rot}%\addtocounter{rot}{1}, \therot

\begin{document}
\title{On the chromatic number of random regular hypergraphs}
\author{Patrick Bennett\thanks{Research supported in part by Simons Foundation Grant \#426894.}\\Department of Mathematics,\\ Western Michigan University\\ Kalamazoo MI 49008\and Alan Frieze\thanks{Research supported in part by NSF Grant DMS1661063 }\\ Department of Mathematical Sciences,\\ Carnegie Mellon University,\\ Pittsburgh PA 15213. }
\date{}

\maketitle

\begin{abstract}
    We estimate the likely values of the chromatic and independence numbers of the random $r$-uniform $d$-regular hypergraph on $n$ vertices for fixed $r$, large fixed $d$, and $n \rightarrow \infty$. 
\end{abstract}

\section{Introduction}
The study of the chromatic number of random graphs has a long history. It begins with the work of Bollob\'as and Erd\H{o}s \cite{BE} and Grimmett and McDiarmid \cite{GM} who determined $\chi(G_{n,p})$, $p$ constant to within a factor 2, w.h.p.~Matula \cite{M} reduced this to a factor of 3/2. Then we have the discovery of martingale concentration inequalities by Shamir and Spencer \cite{SS} leading to the breakthrough by Bollob\'as \cite{B} who determined $\chi(G_{n,p})$ asymptotically for $p$ constant.

The case of $p\to 0$ proved a little more tricky, but {\L}uczak \cite{L1} using ideas from Frieze \cite{F} and \cite{M} determined $\chi(G_{n,p}),p=c/n$ asymptotically for large $c$. {\L}uczak \cite{L2} showed that w.h.p.~$\chi(G_{n,p}),p=c/n$ took one of two values. It was then that the surprising power of the second moment method was unleashed by Achlioptas and Naor \cite{AN}. Since then there has been much work tightening our estimates for the $k$-colorability threshold, $k\ge 3$ constant. See for example Coja-Oghlan \cite{C}. 

Random regular graphs of low degree were studied algorithmically by several authors e.g. Achlioptas and Molloy \cite{AM} and by Shi and Wormald \cite{SW}. Frieze and {\L}uczak \cite{FL92} introduced a way of using our knowledge of $\chi(G_{n,p}),p=c/n$ to tackle $\chi(G_{n,r})$ where $G_{n,r}$ denotes a random $r$-regular graph and where $p=r/n$. Subsequently Achlioptas and Moore \cite{AM} showed via the second moment method that w.h.p.~$\chi(G_{n,r})$ was one of 3 values. This was tightened basically to one value by Coja-Oghlan, Efthymiou and Hetterich \cite{CE}.

For random hypergraphs, Krivelevich and Sudakov \cite{KS98} established the asymptotic chromatic number for $\chi(\cH_r(n,p)$ for $\binom{n-1}{r-1}p$ sufficiently large. Here $\cH_r(n,p)$ is the binomial $r$-uniform hypergraph where each of the $\binom{n}{r}$ possible edges is included with probability $p$. There are several possibilities of a proper coloring of the vertices of a hypergraph. Here we concentrate on the case where a vertex coloring is proper if no edge contains vertices of all the same color. Dyer, Frieze and Greehill \cite{DFG} and Ayre, Coja-Oghlan and Greehill \cite{ACG} established showed that w.h.p.~$\chi(\cH_r(n,p)$ took one or two values. When it comes to what ew denote by $\chi(\cH_r(n,d)$, a random $d$-regular, $r$-uniform hypergraph, we are not aware of any results at all. In this paper we extend the approach of \cite{FL92} to this case:
\begin{theorem}\label{thm:main}
For all fixed $r$ and $\e>0$ there exists $d_0=d_0(r, \e)$ such that for any fixed $d \ge d_0$ we have that w.h.p. 
\begin{equation}
  \abrac{ \frac{ \chi(\mc{H}_r(n, d)) - \rbrac{\frac{(r-1)d}{r \log d}}^{\frac{1}{r-1}}}{\rbrac{\frac{(r-1)d}{r \log d}}^{\frac{1}{r-1}}}} \le \e, \qquad  \abrac{ \frac{ \a(\mc{H}_r(n, d)) - \rbrac{\frac{r \log d}{(r-1)d}}^{\frac{1}{r-1}}n}{\rbrac{\frac{r \log d}{(r-1)d}}^{\frac{1}{r-1}}n}} \le \e
\end{equation}
Here $\a$ refers to the independence number of a hypergraph.
\end{theorem}

\section{Preliminaries}

\subsection{Tools}
We will be using the following forms of Chernoff's bound (see, e.g., \cite{FK}). 

\begin{lemma}[Chernoff bound]
Let $X\sim \bin(n,p)$. Then for all $0<\lambda<np$
\begin{equation}\label{Chernoff}
\Pr(|X-np| \ge \lambda ) \le 2\exp\rbrac{-\frac{ \lambda^2}{3np}}.
\end{equation}
\end{lemma}

\begin{lemma}[McDiarmid's inequality]
Let $X=f(\Vec{Z})$ where $\Vec{Z}= (Z_1, \ldots Z_t)$ and the $Z_i$ are independent random variables. Assume the function $f$ has the property that whenever $\vec{z}, \vec{w}$ differ in only one coordinate we have $|f(\vec{z}) - f(\vec{w})| \le c$. Then for all $\lambda>0$ we have
\begin{equation}\label{Mcdiarmid}
\Pr(|X-\E[X]| \ge \lambda ) \le 2\exp\rbrac{-\frac{ \lambda^2}{2c^2 t}}.
\end{equation}
\end{lemma}

 Bal and the first author \cite{BB21} showed the following.

\begin{theorem}[Claim 4.2 in \cite{BB21}]\label{thm:BB}
Fix $r \ge 3$, $d \ge 2$, and $0 < c < \frac{r-1}{r}.$ Let $z_2$ be the unique positive number such that 
\begin{equation}\label{p1}
\frac{z_2  \sbrac{\rbrac{z_2 + 1}^{r-1} - z_2^{r-1}}}{\rbrac{z_2 + 1}^r - z_2^r} = c
\end{equation}
and let
\beq{p2}{
z_1 = \frac{d}{r\sbrac{\rbrac{z_2 + 1}^r - z_2^r}}.
}
Let $h(x) = x \log x$. If it is the case that   
\begin{equation}\label{eqn:BB}
    h\rbrac{\frac{d}{r}}  + h(d c) + h(d(1-c)) - h(c) -h(1-c) - h(d)  - \frac{d}{r} \log z_1 - dc \log z_2 <0
\end{equation}

then w.h.p.~$\alpha(\mc{H}_r(n, d))< cn$.
\end{theorem}

Krivelevich and Sudakov \cite{KS98} proved the following.

\begin{theorem}[Theorem 5.1 in \cite{KS98}]\label{thm:KS}
For all fixed $r$ and $\e>0$ there exists $d_0=d_0(r, \e)$ such that whenever $D=D(p):=\binom{n-1}{r-1}p \ge d_0$ we have that
\begin{equation*}
  \abrac{ \frac{ \chi(\mc{H}_r(n, p)) - \rbrac{\frac{(r-1)D}{r \log D}}^{\frac{1}{r-1}}}{\rbrac{\frac{(r-1)D}{r \log D}}^{\frac{1}{r-1}}}} \le \e, \qquad \qquad     \abrac{ \frac{ \a(\mc{H}_r(n, p)) - \rbrac{\frac{r \log D}{(r-1)D}}^{\frac{1}{r-1}}n}{\rbrac{\frac{r \log D}{(r-1)D}}^{\frac{1}{r-1}}n}} \le \e
\end{equation*}
with probability at least $1-o(1/n)$.
\end{theorem}

\section{Proof}

In this section we prove Theorem \ref{thm:main}. First we give an overview. We show in Subsection \ref{sec:UBalpha} that the upper bound on $\a$ follows from Theorem \ref{thm:BB} and some straightforward calculations. Then the lower bound on $\chi$ follows as well. Thus we will be done once we prove the upper bound on $\chi$ (since that proves the lower bound on $\a$). This will be in Subsection \ref{sec:UBchi}. For that we follow the methods of Frieze and \L uczak \cite{FL92}.

We will assume $r \ge 3$ since Frieze and \L uczak \cite{FL92} covered the graph case. We will use standard asymptotic notation, and we will use big-O notation to suppress any constants depending on $r$ but not $d$. Thus, for example we will write $r=O(1)$  and $d^{-1} = O(1)$ but not $d=O(1)$. This is convenient for us because even though our theorem is for fixed $d$, it requires $d$ to be sufficiently large.

\subsection{Upper bound on the independence number}\label{sec:UBalpha}

We will apply Theorem \ref{thm:BB} to show an upper bound on  $\alpha(\mc{H}_r(n, d)).$ Fix $\e, r$ (but not $d$) and let $c=c(d):= (1+\e)\rbrac{\frac{r \log d}{(r-1)d}}^{\frac{1}{r-1}}$. Let $z_2$ be as defined in \eqref{p1} and $z_1$ be as defined in \eqref{p2}. We see that 
\begin{lemma}
\[
 z_2 = \frac{c}{1-c} +O\rbrac{c^ {r}}
\]
\end{lemma}
\begin{proof}
After some algebra, we re-write \eqref{p1} as
\[
z_2-\frac{z_2^{r}}{(1+z_2)^{r-1}}=\frac{c}{1-c}.
\]
and the claim follows.
\end{proof}
Now we check \eqref{eqn:BB}.
\begin{align}
   & h\rbrac{\frac{d}{r}}  + h(d c) + h(d(1-c)) - h(c) -h(1-c) - h(d)  - \frac{d}{r} \log z_1 - dc \log z_2 \nn\\
   =& \frac{d}{r} \log \rbrac{\frac{d}{r}}  + dc \log(d c) + d(1-c) \log (d(1-c)) - c \log c -(1-c) \log (1-c) \nn\\
   & \qquad - d\log d  - \frac{d}{r} \log z_1 - dc \log z_2\nn\\
    =& dc \log \sbrac{\frac{c}{(1-c)z_2}}+ \frac dr \log \sbrac{(z_2 + 1)^r-z_2^r} + d \log (1-c) - c \log c - (1-c) \log(1-c). \label{eqn:uppersimp}
\end{align}
Now note that the first term of \eqref{eqn:uppersimp} is
\[
dc \log \sbrac{\frac{c}{(1-c)z_2}} = dc \log \sbrac{\frac{c}{(1-c)\rbrac{\frac{c}{1-c} + O\rbrac{c^{r+1}}}}} = dc \log \sbrac{\frac{1}{1 + O\rbrac{c^{r}}}} = O\rbrac{dc^{r+1}}.
\]
The second term of \eqref{eqn:uppersimp} is
\begin{align*}
 \frac dr \log \sbrac{(z_2 + 1)^r-z_2^r} &=  \frac dr \log \sbrac{\rbrac{\frac{1}{1-c} + O\rbrac{c^{r+1}}}^r-\rbrac{\frac{c}{1-c} + O\rbrac{c^{r+1}}}^r}\\
 &=\frac dr \log \sbrac{\rbrac{\frac{1}{1-c}}^r \rbrac{1-c^r + O\rbrac{c^{r+1}}}}\\
 &=\frac dr \log \rbrac{\frac{1}{1-c}}^r + \frac dr \log \rbrac{1-c^r + O\rbrac{c^{r+1}}}\\
  &= -d \log (1-c) - \frac dr c^r + O\rbrac{dc^{r+1}}.
\end{align*}
The last term of \eqref{eqn:uppersimp} is 
\[
(1-c) \log(1-c) = O(c).
\]
Therefore \eqref{eqn:uppersimp} becomes
\begin{align*}
    &-\frac dr c^r - c \log c +O\rbrac{c+ dc^{r+1}} \\
    =& -c\sbrac{\frac dr c^{r-1} +  \log c}+O\rbrac{c+ dc^{r+1}}\\
     =& -c\sbrac{\frac dr (1+\e)^{r-1} \frac{r \log d}{(r-1)d} +  \log \rbrac{(1+\e)\rbrac{\frac{r \log d}{(r-1)d}}^{\frac{1}{r-1}}}}+O\rbrac{c+ dc^{r+1}}\\
     =& -c\sbrac{ (1+\e)^{r-1} \frac{ \log d}{r-1} -\frac{ \log d}{r-1} +O(\log \log d)}+O\rbrac{c+ dc^{r+1}}\\
     =& -\Omega\rbrac{c \log d}.
\end{align*}
It follows from Theorem \ref{thm:BB} that w.h.p.
\beq{uppalpha}{
\alpha(\mc{H}_r(n, d))\leq(1+\e)\rbrac{\frac{r \log d}{(r-1)d}}^{\frac{1}{r-1}}.
}
\subsection{Upper bound on the chromatic number}\label{sec:UBchi}

Our proof of the upper bound uses the method of Frieze and \L uczak \cite{FL92}. We will generate $\mc{H}_r(n, d)$ in a somewhat complicated way. The way we generate it will allow us to use known results on $\mc{H}_r(n, p)$ due to Krivelevich and Sudakov \cite{KS98}. 

 Set
 \begin{equation}
     m:=\rbrac{\frac {d- d^{1/2} \log d}{r} } n.
 \end{equation}
 
 Let $\mc{H}^*_r(n, m)$ be an $r$-uniform multi-hypergraph with $m$ edges, where each multi-edge consists of $r$ independent uniformly random vertices chosen with replacement.  We will generate $\mc{H}^*_r(n, m)$ as follows. We have $n$ sets (``buckets'' ) $V_1, \ldots V_n$ and a set of $rm$ points $P:=\{p_1, \ldots p_{rm}\}$. We put each point $p_i$ into a uniform random bucket $V_{\f(i)}$ independently. We let $\mc{R}=\{R_1, \ldots, R_m\}$ be a uniform random partition of $P$ into sets of size $r$. Of course, the idea here is that the buckets $V_i$ represent vertices and the parts of the partition $\mc{R}$ represent edges. Thus $R_i$ defines a hyper-edge $\set{\f(j):j\in R_i}$ for $i=1,2,\ldots,m$. We denote the hypergraph defined by $\mc{R}$ by $H_{\mc{R}}$.
 
 Note that since $r \ge 3$ the expected number of pairs of multi-edges in $\mc{H}^*_r(n, m)$ is at most 
 \[
 \binom nr \binom m2 \rbrac{\frac{1}{ \binom nr}}^2 = O\rbrac{\frac{m^2}{n^r}} = O(n^{-1}).
 \]
 Thus, w.h.p.~there are no multi-edges. Now the expected number of ``loops'' (edges containing the same vertex twice) is at most
 \[
 n m \binom r2 \rbrac{\frac 1n}^2 = O(1).
 \]
 Thus w.h.p.~there are at most $\log n$ loops. We now remove all multi-edges and loops, and say that $M$ is the (random) number of edges remaining, where $m - \log n \le M \le m$. The remaining hypergraph is distributed as $\mc{H}(n, M)$, the random hypergraph with $M$ edges chosen uniformly at random without replacement.  Next we estimate the chromatic number of $\mc{H}_r(n, M)$.
 \begin{claim}\label{clm:hnm}
 W.h.p.~we have
 \begin{equation}
  \abrac{ \frac{ \chi(\mc{H}_r(n, M)) - \rbrac{\frac{(r-1)d}{r \log d}}^{\frac{1}{r-1}}}{\rbrac{\frac{(r-1)d}{r \log d}}^{\frac{1}{r-1}}}} \le \frac \e2 , \qquad \qquad     \abrac{ \frac{ \a(\mc{H}_r(n, M)) - \rbrac{\frac{r \log D}{(r-1)D}}^{\frac{1}{r-1}}n}{\rbrac{\frac{r \log D}{(r-1)D}}^{\frac{1}{r-1}}n}} \le \frac \e2 \nn.
\end{equation}
 \end{claim}
 
 \begin{proof}
 We will use Theorem \ref{thm:KS} together with a standard argument for comparing $\mc{H}_r(n, p)$ with $\mc{H}_r(n, m)$. Set $p:=m / \binom nr$ and apply Theorem \ref{thm:KS} with $\e$ replaced with $\e/4$ so we get
\begin{equation}\label{eqn:clm1D}
  \abrac{ \frac{ \chi(\mc{H}_r(n, p)) - \rbrac{\frac{(r-1)D}{r \log D}}^{\frac{1}{r-1}}}{\rbrac{\frac{(r-1)D}{r \log D}}^{\frac{1}{r-1}}}} \le \frac\e4
\end{equation}
with probability at least $1-o(1/n)$. Note that here 
\[
D = \binom{n-1}{r-1}p = \binom{n-1}{r-1}m / \binom nr = rm/n = d-d^{1/2} \log d.
\]
Now since $d, D$ can be chosen to be arbitrarily large and $d=D+O(D^{1/2} \log D)$ we can replace $D$ with $d$ in \eqref{eqn:clm1D} without changing the left hand side by more than $\e / 4$ to obtain 
\begin{equation}\label{eqn:clm1d}
  \abrac{ \frac{ \chi(\mc{H}_r(n, p)) - \rbrac{\frac{(r-1)d}{r \log d}}^{\frac{1}{r-1}}}{\rbrac{\frac{(r-1)d}{r \log d}}^{\frac{1}{r-1}}}} \le \frac \e2
\end{equation}
with probability at least $1-o(1/n)$. But now note that with probability $\Omega(n^{-1/2})$ the number of edges in $\mc{H}_r(n, p)$ is precisely $M$. Thus we have that 
 \begin{equation}
  \abrac{ \frac{ \chi(\mc{H}_r(n, M)) - \rbrac{\frac{(r-1)d}{r \log d}}^{\frac{1}{r-1}}}{\rbrac{\frac{(r-1)d}{r \log d}}^{\frac{1}{r-1}}}} \le \frac \e2 \nn
\end{equation}
with probability at least $1-o(n^{-1/2})$. This proves the first inequality, and the second one follows similarly.
 \end{proof}
Now we will start to transform $\mc{H}_r(n, m)$ to the random regular hypergraph $\mc{H}_r(n, d)$. This transformation will involve first removing some edges from vertices of degree larger than $d$, and then adding some edges to vertices of degree less than $d$.  We define the \tbf{rank} of a point $p_i \in V_j$, to be the number of points $p_{i'} \in V_j$ such that $i'\le i$. We form a new set of points $P' \subseteq P$ and a partition $\mc{R}'$ of $P'$ as follows. For any $R_k \in \mc{R}$ containing a point with rank more than $d$, we delete $R_k$ from $\mc{R}$ and delete all points of $R_k$ from $P$. Note that each bucket contains at most $r$ points of $P'$. Note also that $\mc{R}' $ is a uniform random partition of $P'$. We let $\mc{H}_{\mc{R}'}$ be the natural hypergraph associated with $\mc{R}'$. 

Now we would like to put some more points into the buckets until each bucket has exactly $d$ points, arriving at some set of points $P'' \supseteq P'$. We would also like a uniform partition $\mc{R}''$ of $P''$ into sets of size $r$, and we would like $\mc{R}''$ to have many of the same parts as $\mc{R}'$. We will accomplish this by constructing a sequence $P'_1:=P'\subseteq P'_2 \subseteq \ldots  \subseteq P'_\ell=:P''$ of point sets and a sequence $\mc{R}'_1:=\mc{R}', \mc{R}'_2, \ldots, \mc{R}'_\ell=:\mc{R}''$ where $\mc{R}'_j$ is a uniform random partition of $P'_j$.

We construct $P'_{j+1}, \mc{R}'_{j+1}$ from $P'_{j}, \mc{R}'_{j}$ as follows. Suppose $|\mc{R}'_j|=a$ (in other words $\mc{R}'_j$ has $a$ parts), so $|P'_j| = ra$. $P'_{j+1}$ will simply be $P'_{j}$ plus $r$ {\bf new} points. Now we will choose a random value $K \in \{1, \ldots, r\}$ using the distribution $\P[K=k] = q_k(a)$, where $q_k(a)$ is defined as follows.
\begin{definition}
Consider a random partition of $ra+r$ points into $a+1$ parts of size $r$, and fix some set $Q$ of $r$ points. Then for $1 \le k \le r$, the number $q_k(a)$ is defined to be the probability that $Q$ meets exactly $k$ parts of the partition.
\end{definition} 
We will then remove a uniform random set of $K-1$ parts from $\mc{R}'_{j}$, leaving $Kr$ points in $P'_{j+1}$ which are not in any remaining part of $\mc{R}'_{j}$. We partition those points into $K$ parts of size $r$ such that each part contains at least one new point (each such partition being equally likely), arriving at our partition $\mc{R}'_{j+1}$.

We claim that $\mc{R}'_{j+1}$ is a uniform random partition of $P'_{j+1}$ into parts of size $r$. Indeed, first consider the $r$ new points that are in $P'_{j+1}$ which were not in $P'_{j}$. The probability that a uniform random partition of $P'_{j+1}$ would have exactly $k$ parts containing at least one new point is $q_k$. So we can generate such a random partition as follows: first choose a random value $K$ with $\P[K=k]=q_k$; next we choose a uniform random set of $(K-1)r$ points from $P'_j$; next we choose a partition of the set of points consisting of $P'_{j+1} \sm P'_{j}$ together with the points from $P'_{j}$ we chose in the last step, where the partition we choose is uniformly random from among all partitions such that each part contains at least one point of $P'_{j+1} \sm P'_{j}$; finally, we choose a uniform partition of the rest of the points. In our case this partition of the rest of the points comprises the current partition of the ``unused'' $(a-K+1)r$ points. At the end of this process we have that $\mc{H}_{\mc{R}''}$ is distributed as $\mc{H}_r(n,d)$.
\subsubsection{Bounding the number of low degree vertices in $\mc{H}_{\mc{R'}}$}
We define some sets of buckets. We show that w.h.p.~there are few small buckets i.e few vertices of low degree in the hypregraph $H_{\mc{R'}}$. Let $S_0$ be the buckets with at most $d-3d^{1/2}\log d$ points of $P'$, and let $S_1$ be the buckets with at most $d-2d^{1/2}\log d$ points of $P$. Let $S_2$ be the set of buckets that, when we remove points from $P'$ to get $P$, have at least $d^{1/2}\log d$ points removed. Then $S_0 \subseteq S_1 \cup S_2$. Our goal is to bound the probability that $S_0$ is too large. 

Fix a bucket $V_j$ and let $X \sim \bin\rbrac{rm, \frac{1}{n}}$ be the number of points of $P$ in $V_j$. Then the probability that $V_j$ is in $S_1$ satisfies
\mults{
\Pr[V_j \in S_1] = \Pr\sbrac{X \le d-2d^{1/2}\log d} = \Pr\sbrac{X -\frac{rm}{n} \le -d^{1/2}\log d } \\
\le \exp\sbrac{- \frac{d \log^2 d}{3(d-d^{1/2}\log d)}}= \exp\sbrac{-\Omega\rbrac{ \log^2 d}},
}
where for our inequality we have used the Chernoff bound (Lemma \ref{Chernoff}). Therefore $\E[|S_0|] \le \exp\sbrac{-\Omega\rbrac{ \log^2 d}} n$. Now we argue that $|S_1|$ is concentrated using McDiarmid's inequality (Lemma \ref{Mcdiarmid}). For our application we let $X=|S_1|$ which is a function (say $f$) of the vector $(Z_1, \ldots Z_{rm})$ where $Z_i$ tells us which bucket the $i^{th}$ point of $P$ went into. Moving a point from one bucket to another can only change $|S_1|$ by at most 1 so we use $c=1$. Thus we get the bound
\begin{equation}\label{Az}
\Pr(|X-\E[X]| \ge n^{2/3} ) \le 2\exp\rbrac{-\frac{ n^{4/3}}{2 rm}} = o(1).
\end{equation}

Now we handle $S_2$. For $1 \le j \le n$ let $Y_j$ be the number of parts $R_k \in \mc{R}$ such that $R_k$ contains a point in the bucket $V_j$ as well as a point in some bucket $V_{j'}$ where $|V_{j'}| > d$. Note that if $V_j \in S_2$ then $Y_j \ge d^{1/2} \log d.$ We view $R_k$ as a set of $r$ points, say $\{q_1, \ldots, q_r\}$ each going into a uniform random bucket. Say $q_i$ goes to bucket $V_{j_i}$. The probability that $R_k$ is counted by $Y_j$ is at most
\begin{align*}
   & r\Pr[j_1=j \mbox{ and } |V_{j_1}|>d ] + r(r-1) \Pr[j_1=j \mbox{ and } |V_{j_2}|>d ] \\
   & = \frac{r}{n} \Pr[|V_{j_1}|>d \big| j_1=j]+ \frac{r(r-1)}{n} \Pr[|V_{j_2}|>d \big| j_1=j]\\
   & \le \frac{r^2}{n} \Pr[|V_{j_1}|>d \big| j_1=j]\\
   & \le \frac{r^2}{n} \Pr[\bin(rm-1, 1/n) \ge d] = \frac{r^2}{n}\exp\sbrac{-\Omega\rbrac{ \log^2 d}}.\\
\end{align*}
Thus we have 
\[
\E[Y_j] = m \cdot \frac{r^2}{n}\exp\sbrac{-\Omega\rbrac{ \log^2 d}} \le rd\exp\sbrac{-\Omega\rbrac{ \log^2 d}} = rd^{1/2}\exp\sbrac{-\Omega\rbrac{ \log^2 d}}
\]
and so Markov's inequality gives us
\[
\Pr\sbrac{Y_j \ge d^{1/2} \log d} \le \frac{rd\exp\sbrac{-\Omega\rbrac{ \log^2 d}}}{d^{1/2} \log d} = \exp\sbrac{-\Omega\rbrac{ \log^2 d}}
\]
and so $\E[|S_2|] = n \exp\sbrac{-\Omega\rbrac{ \log^2 d}}.$ We use McDiarmid's inequality once more, this time with $X=|S_2|$. A change in choice of bucket changes $|S_2|$ by at most one and so \eqref{Az} continues to hold. Thus 
\[
|S_0|=n \exp\sbrac{-\Omega\rbrac{ \log^2 d}}.\quad \text{w.h.p.}
\]
\subsubsection{A property of independent subsets of $\mc{H}_r(n,m)$}\label{322}
Fix $1 \le j \le r-1$. Set 
\[
a:= \brac{1+\frac \e2}\rbrac{\frac{r \log d}{(r-1)d}}^{\frac{1}{r-1}}, \qquad \k_j:= \frac {10d}r \binom rj a^j,\qquad p:=\frac{d(r-1)!}{n^{r-1}}.
\]
The expected number of independent sets $A$ in $\mc{H}_r(n, p)$ of size at most $an$ such that there are $\k_j n$ edges each having $j$ vertices in $A$ is at most 
\begin{align*}
   & \sum_{s=1}^{an}\binom{n}{s}(1-p)^{\binom{s}{r}} \binom{\binom{s}{j}\binom{n}{r-j}}{\k_j n }p^{\k_j n}\\
   & \le\sum_{s=1}^{an} \exp\cbrac{ s \log \rbrac{\frac{en}{s}} - \binom sr p + \k_j n \log \rbrac{\frac{e \frac{(an)^j}{j!} \frac{n^{r-j}}{(r-j)!}p}{\k_j n}}}\\
&=\sum_{s=1}^{an} \exp\cbrac{ s \log \rbrac{\frac{en}{s}} - \binom sr p + \k_j n\log\bfrac{ea^j}{10}}\\
   & \le an \cdot  \exp\cbrac{ \sbrac{ \log \rbrac{\frac ea}  - \frac {10d}r \binom rj a^{j-1} \log \rbrac{\frac{10}{e}} }an }\\
   &=o(1/n)
\end{align*}
where the last line follows since as $d \rightarrow \infty$ we have 
\[
\log \rbrac{\frac ea} \sim \frac{1}{r-1} \log d
\]
and 
\[
 \frac {10d}r \binom rj a^{j-1} \log \rbrac{\frac{10}{e}} = \Omega\rbrac{d^{\frac{r-j}{j-1}} \log^{-\frac{j-1}{r-1}} d} \gg \log d.
\]
Thus with probability $1-o(1/n)$, $\mc{H}_r(n, p)$ has a coloring using $(1+\e/2)\rbrac{\frac{(r-1)d}{r \log d}}^{\frac{1}{r-1}}$ colors such that for each color class $A$ and for each $1 \le j \le r-1$ there are at most $\k_j n$ edges with $j$ vertices in $A$. The hypergraph $\mc{H}_r(n, m)$, $m=\binom{n}rp$ will have this property w.h.p..
\subsubsection{Transforming $\mc{H}_{\mc{R}'}$ into $\mc{H}_r(n,d)$}
Now we will complete the transformation to the random regular hypergraph $\mc{H}_r(n, d)$. We are open to the possibility that doing so will render our coloring no longer proper, since this process will involve changing some edges which might then be contained in a color class. We will keep track of how many such ``bad" edges there are and then repair our coloring at the end. 

We have to add at most $(3d^{1/2} \log d+ d \exp\sbrac{-\Omega\rbrac{ \log^2 d}})n < (4d^{1/2} \log d)n$ points, which takes at most as many steps. For each color class $A$ of $\mc{H}_{\mc{R}'}$ define $X_{A, j}=X_{A, j}(i)$ to be the number of edges with $j$ vertices in $A$ at step $i$. We have already established that $X_{A, j}( 0) \le \k_j n$. This follows from Section \ref{322} and the fact that we have removed edges from $\mc{H}(n,m)$ to obtain $\mc{H}_{\mc{R}'}$. Let $\mc{E}_i$ be the event that at step $i$ we have that for each color class $A$ and for each $1 \le j \le r-1$ we have $X_{A, j}( i) \le 2\k_j n$. Then, assuming $\mc{E}_i$ holds, the probability that $X_{A, j}$ increases at step $i$  is at most 
\[
    \sum_{\substack{1\leq k \leq r,\;\; j_\ell \ge 1 \\ j_1 + \cdots + j_k = j}} \prod_{1\le \ell \le k} \frac{2\k_{j_\ell} n}{nd/r}     = \sum_{\substack{1\leq k \leq r,\;\; j_\ell \ge 1 \\ j_1 + \cdots + j_k = j}} \prod_{1\le \ell \le k} 20 \binom r{j_k} a^{j_k}      \le \sum_{\substack{1\leq k \leq r,\;\; j_\ell \ge 1 \\ j_1 + \cdots + j_k = j}} 20^r 2^{r^2} a^j      \le  40^r 2^{r^2} a^j.
\]
Also, the largest possible increase in $X_{A, j}$ in one step is $r$. Thus, the final value of $X_{A, j}$ after at most $(4d^{1/2} \log d)n$ steps is stochastically dominated by $\k_j n + r Y$ where $Y \sim  \bin\big((4d^{1/2} \log d)n, 40^r 2^{r^2} a^j \big)$. An easy application of the Chernoff bound tells us 
\begin{equation}\label{eqn:probbound}
    \Pr\rbrac{Y > 2\E[Y]} \le \exp(-\Omega(n)).
\end{equation}
Note that here 
\[
\frac{2\E[Y]}{\k_j n} = \frac{8d^{1/2}   \log d \cdot 40^r 2^{r^2} a^j n }{10d \binom rj a^j n/r} = O(d^{-1/2} \log d) < 1
\]
for sufficiently large $d$. Thus, using \eqref{eqn:probbound} and the union bound over all color classes $A$, we have w.h.p.~the final value of $X_{A, j}$ is at most $\k_j n + 2\E[Y] \le 2 \k_j n$ for all $1 \le j \le r-1$.

Now we address ``bad'' edges, i.e. edges contained in a color class. Assuming $\mc{E}_i$ holds, the expected number of new edges contained in any color class at step $i$ is at most $ r (40)^r 2^{r^2+2r} a^r = O\rbrac{\rbrac{\frac{\log d}{d}}^\frac{r}{r-1}}$ (because it would have to be one of the colors of one of the vertices we are adding points to). Thus the expected number of bad edges created in $(4d^{1/2} \log d)n$ steps is stochastically dominated by  $Z \sim r\cdot \bin\big((4d^{1/2} \log d)n, O\rbrac{\rbrac{\frac{\log d}{d}}^\frac{r}{r-1}} \big)$. Another easy application of Chernoff shows that w.h.p.~$Z \le 2\E[Z] =O(d^{-1/2} n)$.

We repair the coloring as follows. First we uncolor one vertex from each bad edge, and let the set of uncolored vertices be $U$ where $|U|=u = O\rbrac{d^{-1/2} n}$. Let 
\[
\d := \frac \e2 \rbrac{\frac{(r-1)d}{r \log d}}^{\frac{1}{r-1}}.
\]
We claim that for every $S \subseteq U,|S|=s$, the hypergraph induced on $S$ has at most $\d s / r$ edges. This will complete our proof since it implies that the minimum degree is at most $\d$ and so $U$ can be recolored using a fresh set of $\d$ colors, yielding a coloring of $\mc{H}_r(n, d)$ using at most 
\[
\chi(\mc{H}_r(n, M)) + \d \le \rbrac{1+\frac \e2}\rbrac{\frac{(r-1)d}{r \log d}}^{\frac{1}{r-1}} + \frac  \e 2\rbrac{\frac{(r-1)d}{r \log d}}^{\frac{1}{r-1}} = \rbrac{1+ \e}\rbrac{\frac{(r-1)d}{r \log d}}^{\frac{1}{r-1}}
\]
colors.
 The expected number of sets $S$ with more than $\d s / r$ edges is at most 
\begin{align}
   & \sum_{1 \le s \le u} \binom{n}{s} \binom{\binom{ds}{r}}{\d s / r} \frac{1}{\binom{dn}{r} \binom{dn-r}{r} \ldots \binom{dn-\d s+r}{r}}\nn\\
   & \le \sum_{1 \le s \le u} \rbrac{\frac{ne}{s}}^s \rbrac{\frac{ (dse/r)^r e}{\d s / r}}^{\d s / r} \frac{(r!)^{\d s / r}}{(dn-\d s)^{\d s}}\nn\\
    & \le \sum_{1 \le s \le u} \sbrac{\frac{ne}{s} \rbrac{\frac{dse}{(dn-\d s)r}}^\d \rbrac{\frac{er \cdot r!}{\d s}}^{\d/r}}^s.\label{eqn:S1}
\end{align}
Now for $1 \le s \le \sqrt{n}$ the term in \eqref{eqn:S1} is at most 
\[
\sbrac{O(n) \cdot \rbrac{O(n^{-1/2})}^\d \cdot O(1)}^s = o(1/n)
\]
since $\d$ can be made arbitrarily large by choosing $d$ large. Meanwhile for $\sqrt{n} \le s \le u$ we have that the term in \eqref{eqn:S1} is at most 
\[
\sbrac{O(n^{1/2})  \cdot O(1) \cdot \rbrac{O(n^{-1/2})}^{\d/r}}^s = o(1/n).
\]
Now since \eqref{eqn:S1} has $O(n)$ terms the whole sum is $o(1)$ and we are done. This completes the proof of Theorem \ref{thm:main}.
\section{Summary}
We have asymptotically computed the chromatic number of random $r$-uniform, $d$-regular hypergraphs when proper colorings mean that no edge is mono-chromatic. It would seem likely that the approach we took would extend to other definitions of proper coloring. We have not attempted to use second moment calculations to further narrow our estimates. These would seem to be two natural lines of further research.


\begin{thebibliography}{99}
\bibitem{ACG} P. Ayre, A. Coja-Oghlan and C. Greenhill, Hypergraph coloring up to condensation, {\em Random Structures and Algorithms} 54 (2019) 615 - 652.
%
\bibitem{AM} D. Achlioptas and C. Moore, The Chromatic Number of Random Regular Graphs,  In {\em Jansen, K., Khanna, S., Rolim, J.D.P., Ron, D. (eds) Approximation, Randomization, and Combinatorial Optimization. Algorithms and Techniques. RANDOM APPROX 2004 2004. Lecture Notes in Computer Science, vol 3122. Springer, Berlin, Heidelberg. Approximation, Randomization, and Combinatorial Optimization. Algorithms and Techniques} (2004) 219–228.
%
\bibitem{AN} D. Achlioptas and A. Naor, The two possible values of the chromatic number of a random graph, {\em Annals of Mathematics} 162 (2005) 1335-1351.
%
\bibitem{BB21} D. Bal and P. Bennett, \href{https://arxiv.org/pdf/1603.09232.pdf}{The Matching Process and Independent Process in Random Regular Graphs and Hypergraphs}.
%
\bibitem{B} B. Bollob\'as, The chromatic number of random graphs, {\em Combinatorica} 8  (1988) 49-55.
%
\bibitem{BE} B. Bollob\'as and P. Erd\H{o}s, Cliques in random graphs,
{\em Mathematical Proceedings of the Cambridge Philosophical Society} 80 (1976) 419-427.
%
\bibitem{C} A. Coja-Oghlan, Upper-Bounding the $k$-Colorability Threshold by Counting Covers, {\em Electronic Journal of Combinatorics} 20 (2013).
%
\bibitem{CE} A. Coja-Oghlan, C. Efthymiou and S. Hetterich, On the chromatic number of random regular graphs, {\em Journal of Combinatorial Theory B} 116 (2016) 367-439. 
%
\bibitem{DFG} M. Dyer, A.M. Frieze and C. Greenhill, On the chromatic number of a random hypergraph, {\em Journal of Combinatorial Theorey B} 113 (2015) 68-122.
%
\bibitem{F} A.M. Frieze, On the independence number of random graphs, {\em Discrete Mathematics} 81 (1990) 171-176.
%
\bibitem{FK} A.M. Frieze and M. Karo\'nski, Introduction to Random Graphs, Cambridge University Press, 2015.
%
\bibitem{FL92} A.M. Frieze and T. {\L}uczak, On the independence and chromatic numbers of random regular graphs, {\em Journal of Combinatorial Theory. Series B} 54 (1992) 123-132.
%
\bibitem{GM} G. Grimmett and C. McDiarmid, On colouring random graphs, {\em Mathematical Proceedings of the Cambridge Philosophical Society} 77 (1975) 313-324.
%
\bibitem{KS98} M. Krivelevich and B. Sudakov, The chromatic numbers of random hypergraphs, {\em Random Structures Algorithms} 12 (1998) 381-403.
%
\bibitem{L1} T. {\L}uczak, The chromatic number of random graphs, {\em Combinatorica} 11 (19990) 45-54.
%
\bibitem{L2} T. {\L}uczak, A note on the sharp concentration of the chromatic number of random graphs, {\em Combinatorica} 11 (1991) 295-297.
%
\bibitem{M} D. Matula, Expose-and-Merge Exploration and the Chromatic Number of a Random Graph, {\em Combinatorica} 7 (1987) 275-284.
%
\bibitem{SS} E. Shamir and J. Spencer, Sharp concentration of the chromatic number od random graphs $G_{n,p}$, {\em Combinatorica} 7 (1987) 121-129.
%
\bibitem{SW} L. Shi and N. Wormald, Coloring random regular graphs, {\em Combinatorics, Probability and Computing} 16 (2007) 459-494.
\end{thebibliography}
\end{document}